\documentclass[10pt, twoside, reqno]{amsart}
\usepackage{xcolor}
\usepackage{amstext}
\usepackage{amsfonts}
\usepackage{amssymb}
\usepackage{amsbsy}
\usepackage{latexsym}
\usepackage[T1]{fontenc}
\usepackage{xy}
\usepackage{hhline}
\xyoption{all}
\newcounter{num}[section] %

\newenvironment{theo}
{\refstepcounter{num}%
\bigskip\noindent{\bf Theorem~\arabic{section}.\arabic{num}. }\it}
{\smallskip}

\newenvironment{cor}
{\refstepcounter{num}%
\bigskip\noindent{\bf Corollary~\arabic{section}.\arabic{num}. }\it}
\newenvironment{lemma}
{\refstepcounter{num}%
\bigskip\noindent{\bf Lemma~\arabic{section}.\arabic{num}. }\it}
\newenvironment{theo_with_name}[1]
{\refstepcounter{num}%
\bigskip\noindent{\bf Theorem~\arabic{section}.\arabic{num}} (#1). \it}

\newenvironment{eq}{\begin{equation}}{\end{equation}}

\renewcommand{\Ref}[1]{(\ref{#1})}

\newcommand{\si}{\sigma}
\newcommand{\al}{\alpha}
\newcommand{\be}{\beta}
\newcommand{\ga}{\gamma}
\newcommand{\la}{\lambda}

\newcommand{\ov}[1]{\overline{#1}}
\newcommand{\un}[1]{{\underline{#1}} }

\newcommand{\Sym}{{\mathcal S}}

\newcommand{\upto}{,\ldots ,}
\newcommand{\st}{\,|\,}

\newcommand{\GL}{{\rm GL}}         % GL
\renewcommand{\O}{{\rm O}}         % GL
\newcommand{\Sp}{{\rm Sp}}         % Sp
\newcommand{\Op}{{\rm O}_2^{+}}    % O_2^+
\newcommand{\Om}{{\rm O}_2^{-}}    % O_2^-

\newcommand{\calV}{{\mathcal{V}}}
\newcommand{\calN}{{\mathcal{N}}}
\newcommand{\calU}{{\mathcal{U}}}
\newcommand{\calB}{{\mathcal{B}}}
\newcommand{\calD}{{\mathcal{D}}}
\newcommand{\calT}{{\mathcal{T}}}

\newcommand{\FF}{{\mathbb{F}}}   % base field
\newcommand{\KK}{{\mathbb{K}}}   % some field
\newcommand{\NN}{{\mathbb{N}}}
\newcommand{\ZZ}{{\mathbb{Z}}}   % integers

\renewcommand{\k}{\kappa}

\newcommand{\matr}[4]{\left(
\begin{array}{cc}
#1 & #2 \\ 
#3 & #4 \\ 
\end{array}
\right)}

\newcommand{\vect}[2]{
\begin{pmatrix} #1\\ #2 \end{pmatrix}}

\newcommand{\besep}{\be_{\rm sep}}
\newcommand{\sisep}{\si_{\rm sep}}

\newcommand{\stabp}{\mathrm{Stab}^{+}\!}

\newcommand{\e}{\mathsf{e}}
\newcommand{\tb}[1]{#1}

\newcommand{\mylabel}[1]{}

\begin{document}
\title[Separating invariants for two-dimensional orthogonal groups over finite fields]{Separating invariants for two-dimensional orthogonal groups over finite fields}

\author{Artem Lopatin}

\author{Pedro Antonio Muniz Martins}

\begin{abstract} We described a minimal separating set for the algebra of $\Op(\FF_q)$-invariant polynomial functions of $m$-tuples of two-dimensional vectors over a finite field $\FF_q$. 
 

\end{abstract}

\maketitle

%========================================================
%S1======================================================
\section{Introduction}\label{section_intro}

%========================================================
\subsection{Separating invariants}\label{section_separ} All vector spaces, algebras, and modules are over an arbitrary (possibly finite) field $\FF$ of arbitrary characteristic $p\geq0$ unless otherwise stated.  %By an algebra we always mean an associative algebra with unity.  

Consider an $n$-dimensional vector space $\calV$ over the field $\FF$ with a fixed basis $v_1,\ldots,v_n$.  The coordinate ring $\FF[\calV]  = \FF[ x_1 \upto x_n]$ of $\calV$ is isomorphic to the symmetric algebra $S(\calV^{\ast})$ over the dual space $\calV^{\ast}$, where $x_1 \upto x_n$ is the dual basis for $\calV^\ast$.  Let $G$ be a subgroup of $\GL(\calV) \cong \GL_n(\FF)$. \tb{The space $\calV^{\ast}$ becomes a $G$-module with
\begin{eq}\label{eq_action}
(g \cdot f)(v) = f(g^{-1}\cdot v) 
\end{eq}%
for all $f \in \calV^\ast$ and $v \in \calV$. This action can be extended to the action of $G$ on the algebra $\FF[\calV]$ by the linearity and multipicativity.} The algebra of {\it polynomial invariants} is defined as follows:
$$ \FF[\calV]^G = \{ f \in \FF[\calV] \mid g\cdot f=f  \, \text{ for all  } g \in G \}.  
$$
\noindent{}For an arbitrary infinite extension $\FF\subset \KK$ we can consider any element $f\in\FF[\calV]$ as the \tb{map $ \calV \otimes_{\FF} \KK \to \KK$}. Therefore,  we have 
\begin{eqnarray*} \FF[\calV]^G &= & \{ f \in \FF[\calV] \mid f(g\cdot v)=f(v) \, \text{ for all } g\in G, v \in \calV \otimes_{\FF} \KK \}  \\
& \subset & \{ f \in \FF[\calV] \mid f(g\cdot v)=f(v) \, \text{ for all } g\in G, v \in \calV \} \end{eqnarray*}

In 2002 Derksen and Kemper~\cite{Derksen_Kemper_book} (see~\cite{Derksen_Kemper_bookII} for the second edition) introduced the notion of separating invariants as a weaker concept than generating invariants.   Given a subset $S$ of $\FF[\calV]^G$, we say that elements $u,v$ of $\calV$ {\it are separated by $S$} if  exists an invariant $f\in S$ with $f(u)\neq f(v)$. If  $u,v\in \calV$ are separated by $\FF[\calV]^G$, then we simply say that they {\it are separated}. A subset $S\subset \FF[\calV]^G$ of the invariant ring is called {\it separating} if for any $u, v$ from $\calV$ that are separated we have that they are separated by $S$. We say that a separating set is {\it minimal} if it is minimal w.r.t.~inclusion. Obviously, any generating set is also separating. Denote by $\besep(\FF[\calV]^G)$ the minimal integer $\besep$ such that the set of all invariant polynomials of degree  less or equal to $\besep$ is separating for $\FF[\calV]^G$. Minimal separating sets for different actions were constructed in~\cite{Cavalcante_Lopatin_2020, Domokos_2020, Domokos_2020Add, Lopatin_Ferreira_2023, Kaygorodov_Lopatin_Popov_2018, Kemper_Lopatin_Reimers_2022, Lopatin_Reimers_2021, Lopatin_Zubkov_ANT, Reimers_2020}.

Separating invariants for $\FF[\calV]^G$ in case $\FF = \FF_q$ were studied by Kemper, Lopatin, Reimers in~\cite{Kemper_Lopatin_Reimers_2022}. In particular, a minimal separating set for multi-symmetric polynomials (i.e. the invariants of the symmetric group $\Sym_n$ acting on $V^m$) over $\FF_2$ was found. Note that separating sets for multi-symmetric polynomials over an arbitrary field were studied in~\cite{Lopatin_Reimers_2021}. Recently Domokos and Miklosi~\cite{Domokos_Miklosi_2023} constructed quite small separating set for multisymmetric polynomials over a finite field.

%========================================================
\subsection{Vector invariants}\label{section_vector} 

The algebra of $G$-invariants of vectors is the algebra $\FF[\calV]^G$ with $\calV=V^m$, where $V=\FF^n$, $V^m=V\oplus \cdots \oplus V$ ($m$ times), and the group $G<\GL_n(\FF)$ acts on $V^m$ diagonally: $g\cdot (u_1,\ldots,u_m)=(g u_1,\ldots,g u_m)$ for $g\in G$ and $u_1,\ldots,u_m\in V$.  Over a field $\FF$ of characteristic zero $O_n(\FF)$- and $\Sp_n(\FF)$-invariants of vectors as well as $\GL_n$-invariants of vectors and covectors were described by Weyl~\cite{Weyl_1939}. These results were extended to the case of an arbitrary infinite field by De Concini and Procesi in~\cite{DeConcini_Procesi_1976}, where the characteristic of $\FF$ is odd in case of the orthogonal group. Orthogonal invariants of vectors over an algebraically closed field of characteristic two were studied by Domokos and Frenkel in~\cite{Domokos_Frenkel_2004, Domokos_Frenkel_2005}, but a description of generating invariants is still unknown.

As  about the case of finite fields, in 1911 Dickson~\cite{Dickson_1911} explicitly constructed generators for the algebra of invariants  $\FF_q[V]^{\GL_n(\FF_q)}$. A description of generators for  $\FF_q[V]^{\Sp_n\!(\FF_q)}$ for even $n$ can be found in Section 8.3 of book~\cite{Benson_book_1993} by Benson.  In~\cite{BonnafeKemper_2011} Bonnaf\'e and Kemper  formulated the conjecture on minimal generating set for the algebra of invariants  $\FF_q[V\oplus V^{\ast}]^{\GL_n(\FF_q)}$ of vector and covector, which was confirmed by Chen and Wehlau~\cite{ChenWehlau_2017}. In characteristic two case Chen~\cite{Chen_2018} constructed a minimal generating set for the algebra of orthogonal invariants $\FF_q[V^m]^{\Op(\FF_q)}$ for two dimensional vector space $V$. Kropholler, Mohseni-Rajaei, and Segal~\cite{KrophollerK_2005} described generators and relations between generators for the algebra of invariants $\FF_2[V]^{\O_n(\FF_2,\xi)}$ for the orthogonal group $\O_n(\FF_2,\xi)$ which preserves a non-singular quadratic form $\xi$ on $V$, where $n$ is odd. For $n=4$, a field $\FF$ of odd characteristic, and the quadratic form $\xi=x_1^2-x_2^2+x_3^2-x_4^2$ on $V$ a generating set for $\FF_q[V]^{\O(\FF_q,\xi)}$ was given in~\cite{Chu_2001}.

%========================================================
\subsection{Orthogonal invariants}\label{section_O}

%We consider vector invariants of an orthogonal group in {\it modular} case for two dimensional vector space $V=\FF_q^2$. 
\tb{There are exactly two orthogonal groups for $V=\FF_q^2$: $\Op(\FF_q)$ and $\Om(\FF_q)$ (for example, see Section 2 of~\cite{Chen_2018} and page 213 of~\cite{Neusel_Smith_book}). Note that the order of $\Op(\FF_q)$ is divisible by the characteristic of $\FF_q$ (i.e.,  we have the  {\it modular} case) if and only if $q$ is a $2$-power. The modular case is of more interest, since in the non-modular case many classical tools can be used.  }

For $\al\in\FF_q^{\times}$ denote 
$$\si_{\al}=\matr{0}{\al}{\al^{-1}}{0} \text{ and } \tau_{\al}=\matr{\al}{0}{0}{\al^{-1}}. $$
Then the group $\Op(\FF_q)=\{\si_{\al},\tau_{\al} \st \al\in\FF_q^{\times}\}$ is generated by $\si_1$ and $\tau_{\al}$ for all $ \al\in\FF_q^{\times}$. Given $v\in V$, we denote $v=(v(1),v(2))$. For $m\geq1$ the  coordinate ring of $V^m$ is $\FF_q[V^m]=\FF_q[x_1,\ldots,x_m,y_1,\ldots,y_m]$, where $x_i,y_i\in V^{\ast}$ are defined by
$x_i(\un{v})=v_i(1)$ and $y_i(\un{v})=v_i(2)$ for all $1\leq i\leq m$ and $\un{v}=(v_1,\ldots,v_m)\in V^m$. The action of $\Op(\FF_q)$ on $\FF_q[V^m]$ is given by $\si_1\cdot x_i = y_i$, $\si_1\cdot y_i = x_i$, $\tau_{\al}\cdot x_i=\al^{-1} x_i$, $\tau_{\al}\cdot y_i = \al\, y_i$ for all $1\leq i\leq m$ and $\al\in\FF^{\times}$. For short, we write $\ov{m}=\{1,2,\ldots,m\}$. Given $\un{i}\in \NN^m$, we denote $|\un{i}|=i_1+\cdots+i_m$, where $\NN=\{0,1,2,\ldots\}$. \tb{It is easy to see that the following elements are invariants from $\FF_q[V^m]^{\Op(\FF_q)}$}:
\begin{enumerate}
\item[$\bullet$] $\calN=\Bigl\{  N_i = x_i y_i \,\bigl|\, i\in \ov{m}\Bigr\}$, 

\item[$\bullet$] $\calU=\Bigl\{  U_{ij} = x_i y_j + x_j y_i \,\bigl|\, 1\leq i<j\leq m\Bigr\}$, 

\item[$\bullet$] $\calB=\Bigl\{  B_{\un{i}} = x_1^{i_1} \cdots x_m^{i_m} + y_1^{i_1} \cdots y_m^{i_m} \,\bigl|\,\, \un{i}\in \NN^m, \; |\un{i}|=q-1\Bigr\}$, 

\item[$\bullet$] $\calD=\Bigl\{ d_{IJ}=x_I y_J + x_J y_I \,\bigl|\, \emptyset\neq I<J \subset \ov{m},\, |J|-|I| \text{ is  }0 \text{ or } (q-1)\Bigl\}$, where
\begin{enumerate}
\item[(a)] $x_I=\prod_{i\in I} x_i$ and $y_J=\prod_{j\in J} y_j$;

\item[(b)] $I<J$ stands for the condition that $i<j$ for all $i\in I$ and $j\in J$.
\end{enumerate}
\end{enumerate}

%---------------------------------------------------------
\begin{theo_with_name}{Chen~\cite{Chen_2018}, Theorem 1.1}\label{theo_Chen} In case $p=2$
the set $\calN \cup \calB \cup \calD$ is a minimal generating set for the algebra of invariants $\FF_q[V^m]^{\Op(\FF_q)}$.
\end{theo_with_name}
\medskip

%=========================================================
\subsection{Results}\label{section_results}

In Theorem~\ref{theo_separ} we explicitly described a minimal separating set for $\FF_q[V^m]^{\Op(\FF_q)}$ for all $m>0$. Note that  the constructed separating set is much smaller than the minimal generating set from Theorem~\ref{theo_Chen} in case $p=2$ and $m>1$. We also classified $\Op(\FF_q)$-orbits on $V^m$ in Theorem~\ref{theo_orbit}. As a corollary to Theorem~\ref{theo_separ} in Section~\ref{section_cor} we defined and described $\sisep$ for $\FF_q[V^m]^{\Op(\FF_q)}$ as well as $\besep$.

%=========================================================
\subsection{Notations}\label{section_notations}

Given $v\in V$ and $r\geq 0$, we write $v^{(r)}$ for $(\underbrace{v,\ldots,v}_r)\in V^r$. We say that $\un{v}\in V^m$ has no zeros if $v_i$  is non-zero for all $i$. If for $\un{u},\un{v}\in V^m$ there exists $g\in \Op(\FF_q)$ such that $g\cdot \un{u}=\un{v}$, then we write $\un{u}\sim \un{v}$.  Given $v\in V$, we write $\stabp(v)$ for the stabilizer of $v$ in the group $\Op(\FF_q)$.

%Assume that $\algA$ is  $\mathcal{O}(M_n^m)^{\GL_n}$ or $\mathcal{O}(\N_n^m)^{\GL_n}$. We say that an $\NN$-homogeneous invariant $f\in \algA$ is {\it decomposable} and write $f\equiv0$ if $f$ is a polynomial in $\NN$-homogeneous invariants of $\algA$ of strictly lower degree. If $f$ is not decomposable, then we say that $f$ is {\it indecomposable} and write $f\not\equiv0$. In case $f-h\equiv0$ we write $f\equiv h$. Obviously, if $f\equiv0$ for $f\in \mathcal{O}(M_n^m)^{\GL_n}$, then $\Psi(f)\equiv0$ in $\mathcal{O}(\N_n^m)^{\GL_n}$.  

%For a monomial $c\in \FF[M_n^m]$ denote by $\deg{c}\in \NN$ its {\it degree} and by $\mdeg{c}\in \NN^m$ its {\it multidegree}, where $\NN$ stands for the set of non-negative integers. Namely, $\mdeg{c}=(t_1,\ldots,t_d)$, where $t_k$ is the total degree of the monomial $c$ in $x_{ij}(k)$, $1\leq i,j\leq n$, and $\deg{c}=t_1+\cdots+t_m$.   
%We say that a multidegree $(t_1,\ldots,t_m)$ is less than a multidegree $(k_1,\ldots,k_m)$ if $t_i\leq k_i$ for all $i$ and $t_j< k_j$ for some $j$.

%=========================================================
%=========================================================

%=========================================================
\section{Classification of $\Op(\FF_q)$-orbits}\label{section_orbit} 

Given $\al\in\FF_q$, for short we write  $\e_{\al}$ for $\vect{1}{\al}\in V$. For $\al\in \FF_q^{\times}$ denote by $\Omega_{\al}$ any set of representatives of orbits of $\ZZ_2\simeq \{\tau_1,\si_{\al^{-1}}\}$ on the set
$$S_{\al}=\left\{\vect{\be}{\ga} \,\Bigl|\, \be,\ga \in\FF_q, \; \al\be\neq \ga \right\} = V \, \backslash\, \FF_q \e_{\al}.$$
Note that each of these orbits contains exactly two elements.  The following remark is  trivial.

%---------------------------------------------------------
\begin{remark}\label{remark_stabi}
\begin{enumerate}
\item[1.] Assume that $\al,\be\in\FF_q$ are not both equal to zero. Then
\begin{enumerate}
\item[$\bullet$] $\stabp{\vect{\al}{\be}} = \{\tau_1\}$ in case $\al=0$ or $\be=0$,

\item[$\bullet$] $\stabp{\vect{\al}{\be}} = \{\tau_1, \si_{\al\be^{-1}}\}$ in case $\al$ and $\be$ are non-zero.
\end{enumerate}

\item[2.] For every non-zero $u,v\in V$ with $\FF_q u\neq \FF_q v$ we have that $\stabp(u)\bigcap \stabp(v) = \{\tau_1\}$.
\end{enumerate}
\end{remark}

%---------------------------------------------------------
\begin{remark}\label{remark_class}
If $v\in V$ is non-zero, then either $v\sim \e_0$ or $v\sim \e_{\al}$, where $\al\in \FF_q^{\times}$.
\end{remark}
\begin{proof} Acting by $\si_1$, we can assume that $v=\vect{\be}{\ga}$ with $\be\in\FF^{\times}$ and $\ga\in\FF_q$. Acting by $\tau_{\be^{-1}}$ on $v$, we obtain the required.
\end{proof}

%---------------------------------------------------------
\begin{lemma}\label{lemma_equiv} Assume that $u,v,w,u',v',w'\in V$ and $\al\in\FF_q^{\times}$. 
\begin{enumerate}
\item[1.] If $\left(\e_{0},u\right) \sim \left(\e_{0},u'\right)$, then $u=u'$.

\item[2.] If $\left(\e_{\al},v\right) \sim \left(\e_{\al},v'\right)$ and $v\in \FF_q \e_{\al}$, then $v=v'$.

\item[3.] If $\left(\e_{\al},w\right) \sim \left(\e_{\al},w'\right)$ and $w,w'\in \Omega_{\al}$, then $w=w'$.

\item[4.] If $\left(\e_{\al},w,u\right) \sim \left(\e_{\al},w',u'\right)$ and $w,w'\in \Omega_{\al}$, then $w=w'$ and $u=u'$.
\end{enumerate}
\end{lemma}
\begin{proof} 
\noindent{\bf 1.} Since the stabilizer of $\e_{0}$ is $\{\tau_1\}$ by part 1 of Remark~\ref{remark_stabi}, we obtain $u=u'$.

\medskip
\noindent{\bf 2.}  Since the stabilizers of $\e_{\al}$ and $v$ are the same, we obtain $v=v'$.

\medskip
\noindent{\bf 3.} Since the stabilizer of $\e_{\al}$ is $\{\tau_1,\si_{\al^{-1}}\}$ by part 1 of Remark~\ref{remark_stabi}, the definition of $\Omega_{\al}$ implies that $w=w'$.

\medskip
\noindent{\bf 4.} By part 3 we have $w=w'$. By part 2 of Remark~\ref{remark_stabi}, the intersection of stabilizers of $\e_{\al}$ and $w$ is $\{\tau_1\}$. Therefore, $u=u'$. 
\end{proof}

%---------------------------------------------------------
\begin{theo}\label{theo_orbit} Assume that the characteristic of $\FF_q$ is arbitrary and $m>0$. Then each $\Op(\FF_q)$-orbit on $V^m$ contains one and only one element, which is called $\Op(\FF_q)$-canonical, of the following type: 
\begin{enumerate}
\item[(0)] $(0,\ldots,0)$; 

\item[(a)] $\left(0^{(r)},\vect{1}{0}, u_1,\ldots,u_t\right)$, where $r\geq0$, $u_1,\ldots, u_t\in V$;

\item[(b)] $\left(0^{(r)},\vect{1}{\al}, \vect{\be_1}{\al\be_1},\ldots,  \vect{\be_s}{\al \be_s}\right)$, where $r,s\geq0$, $\al\in \FF_q^{\times}$, $\be_1,\ldots,\be_s\in \FF_q$;

\item[(c)] $\left(0^{(r)},\vect{1}{\al}, \vect{\be_1}{\al\be_1},\ldots,  \vect{\be_s}{\al\be_s},w,u_1,\ldots,u_t\right)$, where 
\begin{enumerate}
\item[$\bullet$] $r,s,t\geq0$,
\item[$\bullet$] $\al\in \FF_q^{\times}$, $\be_1,\ldots,\be_s\in \FF_q$,
\item[$\bullet$] $w\in \Omega_{\al}$, $u_1,\ldots, u_t\in V$.
\end{enumerate}
\end{enumerate}
\end{theo}
\begin{proof} \noindent{\bf 1.} At first, we show that any $\un{v}\in V^m$ lies in an orbit containing an element from the formulation of the theorem. Obviously, \tb{one can reduce to the case when}  $\un{v}$ has no zeros. Moreover, by Remark~\ref{remark_class}, we assume that $v_1=\e_{\al}$ for some $\al\in\FF_q$. If $\al=0$, then case (a) holds. 

Assume that $\al\neq 0$. Denote 
$$s=\max\{0\leq i\leq m-1 \st v_2,\ldots,v_{\tb{i}+1} \not\in S_{\al}\}.$$
Note that for any non-zero $u\in V$ the conditions $u\not\in S_{\al}$ and $u\in \FF_q \e_{\al}$ are equivalent. Therefore, case (b) holds for $s=m-1$. 

Assume that $s<m-1$. By Remark~\ref{remark_stabi} we have that 
$\stabp(\e_{\al}) = \stabp(v_2) = \cdots= \stabp(v_{s\tb{+1}}) =  
\{\tau_1,\si_{\al^{-1}}\}$. Therefore, acting by  $\{\tau_1,\si_{\al^{-1}}\}$ we can assume that $v_{s+2}\in \Omega_{\al}$. Hence, case (c) holds.

\medskip
\noindent{\bf 2.}  To prove uniqueness, consider some $\Op(\FF_q)$-canonical elements $\un{v},\un{v}'\in V^m$ satisfying the condition
$\un{v} \sim \un{v}'$. Since for each $i$ we have $v_i=0$ if and only if $v'_i=0$, without loss of generality we can assume that $\un{v},\un{v}'$ do not have zeros. By the definition of polynomial invariants for every $f\in \calN \cup \calU \cup \calB \cup \calD$ we have that $f(\un{v})=f(\un{v}')$.  

Since $v_1(1)=v'_1(1)=1$, the condition $N_1(\un{v})=N_1(\un{v}')$ implies that $v_1=v'_1$. 

If $v_1=v'_1=\e_{0}$, then applying part 1 of Lemma~\ref{lemma_equiv} to pares $(v_1,v_i)\sim (v'_1,v'_i)$ we obtain that $v_i=v'_i$, where $2\leq i\leq m$.

Assume $v_1=v'_1=\e_{\al}$ for some $\al\in\FF_q^{\times}$. As in the proof of part 1, $s=\max\{0\leq i\leq m-1 \st v_2,\ldots,v_{\tb{i}+1} \not\in S_{\al}\}$ and  $s'=\max\{0\leq i\leq m-1 \st v'_2,\ldots,v'_{\tb{i}+1} \not\in S_{\al}\}$. Applying part 2 of Lemma~\ref{lemma_equiv} to pares $(v_1,v_i)\sim (v'_1,v'_i)$ we obtain that $s=s'$ and $v_i=v'_i$ for all $\tb{2}\leq i\leq s\tb{+1}$. In particular, if $s=m-1$, then $\un{v}=\un{v}'$ have both type (b).

Assume that $s\leq m-2$, i.e., $\un{v}$, $\un{v}'$ have both type (c). Applying part 3 of Lemma~\ref{lemma_equiv} to \tb{pairs} $(v_1,v_{s+2})\sim (v'_1,v'_{s+2})$ we obtain $v_{s+2}=v'_{s+2}$, since $v_{s+2},v'_{s+2}\in \Omega_{\al}$. 

If $s\leq m-3$, then for every $s+3\leq i\leq m$ we apply part 4 of Lemma~\ref{lemma_equiv} to the triples $(v_1,v_{s+2}, v_i)\sim (v'_1,v'_{s+2},v'_i)$ to obtain $v_i=v'_i$. Hence $\un{v}=\un{v}'$.
\end{proof}
\medskip

%---------------------------------------------------------
\begin{cor}\label{cor_orbit}
\tb{ The number of $\Op(\FF_q)$-orbits on $V^m$ is equal to 
$$\k= \frac{(q^m+1)(q^m+q-2)}{2(q-1)}.$$ }
\end{cor}
\begin{proof} Denote by $\k_1$, $\k_2$, $\k_3$, respectively, the number of orbits from Theorem~\ref{theo_orbit} of type (a), (b), (c), respectively. We have
$$\k_1=\sum_{t=0}^{m-1} q^{2t} = \frac{q^{2m} -1 }{q^2-1} \quad \text{ and } \quad
\k_2=\sum_{s=0}^{m-1} (q-1)q^{s} = q^m-1.$$%
Note that $|\Omega_{\al}| = |S_{\al}| / 2 = q(q-1)/2$. Hence, for $m\geq 2$ we have
$$\k_3 = \sum (q-1)q^s \,\frac{q(q-1)}{2}\, q^{2t} = \frac{q(q-1)^2}{2} A \quad
\text{ for} \quad A = \sum q^{s+2t},$$
where both sums range over all $s,t\geq0$ with $s+t\leq m-2$. Rewriting $A$ as
$$A=\sum_{t=0}^{m-2}\biggr  (\sum_{s=0}^{m-t-2} q^s \biggr) q^{2t} = \sum_{t=0}^{m-2} \frac{q^{m+t-1} - q^{2t}}{q-1} = 
\frac{1}{q-1} \biggr(  \frac{q^{m-1}-1}{q-1} q^{m-1} - \frac{q^{2(m-1)}-1}{q^2-1}\biggr)$$
$$ = \frac{1}{(q-1)^2 (q+1)}\biggr( q^{2m-1} - q^{m} - q^{m-1} +1 \biggr),$$%
we can see that
\begin{eq}\label{eq_k3}
\k_3 = \frac{q}{2(q+1)} (q^m - 1)(q^{m-1} -1) \;\;\text{ for }\;\; m\geq2
\end{eq}%
Note that in case $m=1$ we have $\k_3=0$; therefore, formula~(\ref{eq_k3}) also holds for $m=1$. Finally, 
$$\k= 1 + \k_1 + \k_2 + \k_3 = \frac{1}{2(q^2 -1)} \biggr( q^{2m+1} + q^{2m} + q^{m+2} - q^{m} + q^2 - q - 2 \biggr) = \frac{(q^m+1)(q^m+q-2)}{2(q-1)}.$$
\end{proof}

%=========================================================
\section{$\Op(\FF_q)$-invariants}\label{section_invOp} 

Denote by $\calT_m$ the following set: 
$$N_i=x_i y_i,\quad  T_i = x_i^{q-1} + y_i^{q-1} \quad (1\leq i\leq m), $$
$$U_{ij}=x_i y_j + x_j y_i, \quad 
H_{ij} = x_i x_j^{q-2} + y_i y_j^{q-2} \quad (1\leq i<j\leq m). $$
\noindent{}Denote by $\calT_m^{(2)}$ the following subset of $\calT_m$:
$$N_i,\;T_i  \quad (1\leq i\leq m),\quad U_{ij} \quad (1\leq i<j\leq m).$$
\noindent{}The consideration of $\calT_m^{(2)}$ is motivated by the fact that 
\begin{eq}\label{eq_H}
H_{ij} = T_i\quad \text{ in case } \quad q=2.
\end{eq}%
Note that $\calT_1=\calT_1^{(2)}$. Since $T_i = B_{\un{i}}$ for $\un{i}=(0,\ldots,0,q-1,0,\ldots,0)$ with the only non-zero entry in position $i$ and $H_{ij}= B_{\un{j}}$ for $\un{j}=(0,\ldots,0,1,0\ldots,0,q-2,0,\ldots,0)$ with the only non-zero entries in positions $i$ and $j$, all elements of $\calT_m$ lie in $\FF_q[V^m]^{\Op(\FF_q)}$. In case some $\un{v},\un{v}'\in V^m$ are fixed and $f(\un{v})=f(\un{v}')$ holds for some $f\in \calT_m$, we denote this equality by $(f)$. As an example, see below the proof of Lemma~\ref{lemma_separ1}. 

The following remark follows from Theorem~\ref{theo_Chen} (see also Proposition 2.3, Example 1.5, Example 7.1  from~\cite{Chen_2018}).

%---------------------------------------------------------
\begin{remark}\label{remark_generators} The algebra of invariants $\FF_q[V^m]^{\Op(\FF_q)}$ is minimally generated by
\begin{enumerate}
\item[1.] $N_1$, $T_1=B_{(q-1)}=x_1^{q-1} + y_1^{q-1}$, in case $m=1$. Note that these elements are algebraically independent.

\item[2.] $N_1$, $N_2$, $U_{12}$, $B_{(i,q-i-1)}=x_1^{i} x_2^{q-i-1} + y_1^{i} y_2^{q-i-1}$ for all $1\leq i\leq q-1$, in case $m=2$. Note that here $\calU=\calD$. 

\item[3.] $N_1$, $N_2$, $N_3$, $U_{12}$, $U_{13}$, $U_{23}$, $B_{(i,j,q-i-j-1)}=x_1^i x_2^j x_3^{q-i-j-1} + y_1^i y_2^j y_3^{q-i-j-1}$ for all $1\leq i,j\leq q-1$, in case $m=3$. Note that here $\calU=\calD$. 
\end{enumerate}
\end{remark}

%---------------------------------------------------------
\begin{lemma}\label{lemma_separ1}
The set $\calT_1$ is a minimal separating set for $\FF_q[V]^{\Op(\FF_q)}$.
\end{lemma}
\begin{proof} \noindent{\bf 1.} Assume that $v,v'\in V$ are not separated by $\calT_1$. Without loss of generality, we can assume that $v$ and $v'$ are $\Op(\FF_q)$-canonical (see Theorem~\ref{theo_orbit}). 

Let $v=0$ and $v'= \vect{\al}{\be}$. Since $N_1(v')=T_1(v')=0$, we obtain $\al\be=0$ and $\al^{q-1}+\be^{q-1}=0$. Hence $\al=\be=0$.

Assume that $v$ and $v'$ are non-zero. Then $v=\e_{\al}$ and $v'=\e_{\al'}$ for some $\al,\al'\in\FF_q$. Since $(N_1)$, we obtain $\al=\al'$. 

\medskip
\noindent{\bf 2.} The minimality follows from the facts that $0\not\sim\e_{0}$ are not separated by $\{N_1\}$ and \tb{$e_{\al}\not\sim\e_{\be}$ are not separated by $\{T_1\}$, where $\al\neq\be$ lie in $\FF_q$ and 
\begin{enumerate}
\item[$\bullet$] $\al,\be$ are non-zero in case $q>2$; 
\item[$\bullet$] $\al=0$, $\be=1$ in case $q=2$. 
\end{enumerate}
}
\end{proof}

%---------------------------------------------------------
\begin{lemma}\label{lemma_separ2}
The set 
\begin{enumerate}
\item[$\bullet$] $\calT_2^{(2)}$, in case $q=2$; 

\item[$\bullet$] $\calT_2$, \;\,  in case $q>2$; 
\end{enumerate}
is a minimal separating set for $\FF_q[V^2]^{\Op(\FF_q)}$.
\end{lemma}
\begin{proof} \noindent{\bf 1.} Assume that the set from the formulation of the lemma is not separating. Then by formula~\Ref{eq_H} we can assume in both cases that there are $\un{v},\un{v}'\in V^2$, which are not separated by $\calT_2$ but $\un{v}\not\sim\un{v}'$.  Without loss of generality, we can assume that $\un{v}$ and $\un{v}'$ are $\Op(\FF_q)$-canonical (see Theorem~\ref{theo_orbit}). Moreover, we can assume that $\un{v}$ and $\un{v}'$ have no zeros, since $v_i=0$ if and only if $v'_i=0$ for every $i=1,2$ (see Lemma~\ref{lemma_separ1}).

Applying Lemma~\ref{lemma_separ1} to $v_1$ and $v'_1$ we obtain that $v_1=v'_1=\e_{\al}$ for some $\al\in\FF$. 
%In particular, $\un{v}$ and $\un{v}'$ have one and the same type. 
Denote $v_2=\vect{\be}{\ga}$ and $v'_2=\vect{\be'}{\ga'}$.

Assume $\al=0$. Equalities $(U_{12})$ and $(T_2)$ imply that  $\ga=\ga'$ and $\be^{q-1}=(\be')^{q-1}$, respectively.
If $q=2$, then $\be=\be'$; a contradiction. If $q>2$, then $(H_{12})$ implies that $\be^{q-2}=(\be')^{q-2}$ and we obtain $\be=\be'$; a contradiction.

%Assume $\al\neq 0$ and $q=2$. Then $\al=1$. Obviously, equalities $(N_2)$ and $(T_2)$ imply %that either $v_2=v'_2$ or $v_2,v'_2\in \left\{\e_0, \vect{0}{1}\right\}=S_1$. Since $\al=1$, %the types of $\un{v}$ and $\un{v}'$ are (b) or (c). The condition $v_2,v'_2\not\in \FF_q v_1$ %implies that the types of $\un{v}$ and $\un{v}'$  are both (c) and $v_2,v'_2\in\Omega_1$. Since % $\si_1\cdot \e_0 = \vect{0}{1}$, the set $\Omega_1$ contains a unique element and  $v_2=v_2'$; %a contradiction.

Assume $\al\neq 0$. Lemma~\ref{lemma_separ1} implies that $v_2\sim v'_2$. Therefore, there exists $\la\in\FF_q^{\times}$ such that 
$$v_2' = \tau_{\la}\cdot v_2 = \vect{\la \be}{\la^{-1} \ga}  \quad \text{ or }\quad
v_2' = \si_{\la}\cdot v_2 = \vect{\la \ga}{\la^{-1} \be}.
$$%
\noindent{}Consider the equality $(U_{12})$: $\ga + \al\be = \ga' + \al\be'$.

Assume $v'_2=\tau_{\la}\cdot v_2$. Then $(U_{12})$ implies  
$$\ga + \al\be = \ga\la^{-1} + \al\be\la.$$
Thus  $\ga(1 - \la^{-1}) = \al\be\la(1 - \la^{-1})$.  We have $\la \neq 1$ and $\ga=\al\be\la$, since otherwise  $v_{2}'= v_{2}$; a contradiction.  Note that $\be\neq 0$, since otherwise $\ga = 0$ and $v_2 = 0$;  a contradiction.  Hence $\la = \ga\al^{-1}\be^{-1}$ and  $v_{2}'=\vect{\al^{-1}\ga}{\al\be}$. Remark~\ref{remark_stabi} implies that $\un{v}' =\si_{\al^{-1}}\cdot \un{v}$;  a contradiction.  

Assume $v'_2=\si_{\la}\cdot v_2$. Then $(U_{12})$ implies 
$$\ga +\al\be = \be\la^{-1} + \al\ga\la$$
Thus $\ga\la (\la^{-1} - \al) = \be (\la^{-1} - \al)$. In case $\la=\al^{-1}$ we have $v_{2}' = \vect{\al^{-1} \ga}{\al\be}$ and  $\un{v}' = \si_{\al^{-1}} \cdot \un{v}$; a contradiction. In case $\la\neq\al^{-1}$ we have $\ga\la =\be$. Note that $\ga \neq 0$, since otherwise $\be = 0$ and $v_2=0$; a contradiction. Hence $\la = \be\ga^{-1}$ and $v_{2}' =\vect{\be}{\ga}=v_2$; a contradiction.

\medskip
\noindent{\bf 2.} In case $q=2$ we have that $(\e_0,\e_0)\not\sim \left(\e_0,\vect{0}{1}\right)$ are not separated by $\calT^{(2)}_2\backslash\{U_{12}\}$.

In case $q>2$ consider some $\al\in \FF_q\backslash \{0,1\}$. Then $(\e_0,\e_0)\not\sim \left(\e_0,\vect{\al}{0}\right)$ are not separated by $\calT_2\backslash\{H_{12}\}$. Moreover, $\left(\e_0,\vect{0}{1}\right)\not\sim \left(\e_0,\vect{0}{\al}\right)$ are not separated by $\calT_2\backslash\{U_{12}\}$. Therefore, the minimality is proven. 
\end{proof}

%---------------------------------------------------------
%\begin{lemma}\label{lemma_separ3}
%The set 
%\begin{enumerate}
%\item[$\bullet$] $\calT_3^{(2)}$, in case $q=2$; 
%
%\item[$\bullet$] $\calT_3$, \;\, in case $q>2$; 
%\end{enumerate}
%is a minimal separating set for $\FF_q[V^3]^{\Op(\FF_q)}$.
%\end{lemma}
%\begin{proof} 
%...
%\end{proof}

%---------------------------------------------------------
\begin{theo}\label{theo_separ}
Assume that the characteristic of $\FF_q$ is arbitrary. Then the set 
\begin{enumerate}
\item[$\bullet$] $\calT_m^{(2)}$, in case $q=2$; 

\item[$\bullet$] $\calT_m$, \; in case $q>2$; 
\end{enumerate}
is a minimal separating set for $\FF_q[V^m]^{\Op(\FF_q)}$ for all $m>0$.
\end{theo}
\begin{proof}  \noindent{\bf 1.} Assume that the set from the formulation of the lemma is not separating. Then by Lemma~\ref{lemma_separ2} we have that $m>2$. Moreover, by formula~\Ref{eq_H} we can assume in both cases that there are $\un{v},\un{v}'\in V^m$, which are not separated by $\calT_m$ but $\un{v}\not\sim\un{v}'$.  Without loss of generality, we can assume that $\un{v}$ and $\un{v}'$ are $\Op(\FF_q)$-canonical (see Theorem~\ref{theo_orbit}). Moreover, we can assume that $\un{v}$ and $\un{v}'$ have no zeros, since $v_i=0$ if and only if $v'_i=0$ for every $1\leq i\leq m$ (see Lemma~\ref{lemma_separ1}). 

Applying Lemma~\ref{lemma_separ1} to $v_1$ and $v'_1$ we obtain that $v_1=v'_1=\e_{\al}$ for some $\al\in\FF$. Given $2\leq i\leq m$, we apply Lemma~\ref{lemma_separ2} to the pairs $(v_1,v_i)$ and $(v'_1,v'_i)$ to obtain that 
\begin{eq}\label{eq_equiv}
(\e_{\al},v_i)\sim (\e_{\al},v'_i).
\end{eq}

Assume $\al=0$.  Then equivalence~\Ref{eq_equiv} together with part 1 of Lemma~\ref{lemma_equiv} implies that $v_i=v_i'$ for all  $2\leq i\leq m$. Thus $\un{v}=\un{v}'$; a contradiction.

Assume $\al\neq0$. Consider some  $2\leq i\leq m$.  
\tb{ If $v_i\in \FF_q \e_{\al}$, then equivalence~\Ref{eq_equiv} together with part 2 of Lemma~\ref{lemma_equiv} implies that $v_i=v_i'$. Similarly we obtain that $v_i=v_i'$ in case $v'_i\in \FF_q \e_{\al}$. Thus $\un{v}$ is $\Op(\FF_q)$-canonical of type (b) if and only if  $\un{v}'$ is $\Op(\FF_q)$-canonical of type (b). Moreover, in this case we obtain that $\un{v}=\un{v}'$; a contradiction. }

\tb{Therefore,  $\un{v}$ and $\un{v}'$ are both $\Op(\FF_q)$-canonical of type (c). Moreover,
$$\,\un{v}\,=(\e_{\al}, \be_1 \e_{\al}, \ldots,\be_s \e_{\al},w\,,u_1\,,\ldots,u_t),$$
$$\un{v}'=(\e_{\al}, \be_1 \e_{\al}, \ldots,\be_s \e_{\al},w',u'_1,\ldots,u'_t),$$
where $s,t\geq0$, $\be_1,\ldots,\be_s\in\FF_q$, $w,w'\in\Omega_{\al}$, $u_1,\ldots,u_t\in V$ and $u'_1,\ldots,u'_t\in V$. Since $(\e_{\al},w)\sim (\e_{\al},w')$ by equivalence~\Ref{eq_equiv}, part 3 of Lemma~\ref{lemma_equiv} implies that $w=w'$. In case $t=0$ we obtain $\un{v}=\un{v}'$; a contradiction. Hence $t>0$.}

\tb{Since $\un{v}\not\sim\un{v}'$, there exists $1\leq i\leq t$ with $u_i\neq u'_i$.  Equivalence~\Ref{eq_equiv} implies $(\e_{\al},u_i)\sim (\e_{\al},u'_i)$. If $u_i$ or $u'_i$ lies in $\FF_q \e_{\al}$, then as above we obtain $u_i=u'_i$; a contradiction. Therefore, $u_i,u'_i\in S_{\al}$. Part 1 of Remark~\ref{remark_stabi} implies that $u'_i=\si_{\al^{-1}} u_i$. Denote 
$$w=\vect{\be}{\ga}\;\; \text{ and } \;\; u_i=\vect{\la_1}{\la_2},$$
where $\be,\ga,\la_1,\la_2\in\FF_q$ and $\al\be\neq \ga$, $\al \la_1\neq \la_2$. Since $u'_i=\vect{\al^{-1}\la_2}{\al \la_1}$, equality $(U_{s+2,s+i+2})$ implies
$$\be \la_2 + \ga \la_1 = \al\be\la_1 + \al^{-1}\ga\la_2.$$%
Thus $\be(\la_2 - \al\la_1) = \al^{-1} \ga (\la_2 - \al\la_1)$. Since $\la_2\neq \al \la_1$, we have $\be=\al^{-1} \ga$; a contradiction.}

%Assume $\al\neq0$. Consider some  $2\leq i\leq m$.  Since $\un{v}$ and $\un{v}'$ are $\Op(\FF_q)$-canonical, we have that $v_i,v'_i$ belong to $\FF_q \e_{\al} \sqcup \Omega_{\al}$. If $v_i,v'_i\in \FF_q \e_{\al}$, then equivalence~\Ref{eq_equiv} together with part 2 of Lemma~\ref{lemma_equiv} implies that $v_i=v_i'$. If $v_i,v'_i\in \Omega_{\al}$, then equivalence~\Ref{eq_equiv} together with part 3 of Lemma~\ref{lemma_equiv} imply that $v_i=v_i'$. If $v_i\in \FF_q\e_{\al}$ and $v'_i\in \Omega_{\al}$, then the pairs $(\e_{\al},v_i)$ and $(\e_{\al},v'_i)$ are both $\Op(\FF_q)$-canonical of types (b) and (c), respectively; a contradiction to equivalence~\Ref{eq_equiv} and Theorem~\ref{theo_orbit}. The case of  $v_i\in \Omega_{\al}$ and $v'_i\in \FF_q\e_{\al}$ is the same as above. Therefore, we obtain that $\un{v}=\un{v}'$; a contradiction. 

\medskip
\noindent{\bf 2.} 
The minimality follows immediately from the minimality in case $m=2$ (see Lemma~\ref{lemma_separ2}) and the fact that all elements of $\calT_m^{(2)}$ and $\calT_m$ depends on one or two vectors.
\end{proof}

%=========================================================
\section{Corollaries}\label{section_cor}

As in Section~\ref{section_separ}, assume that $\calV$ is an $n$-dimensional vector space over $\FF$,  $G$ is a subgroup of $\GL(\calV)$.  The  coordinate ring of $\calV^m$ is $\FF_q[\calV^m]=\FF_q[x_{1,1},\ldots,x_{m,1},\ldots, x_{1,n},\ldots,x_{m,n}]$, where $x_{i,j}\in (V^m)^{\ast}$ is defined by
$x_{i,j}(\un{v})=v_i(j)$  for all $1\leq i\leq m$, $1\leq j\leq n$ and $\un{v}=(v_1,\ldots,v_m)\in V^m$. We say that an $m_0$-tuple $\un{i}\in\NN^{m_0}$ is {\it $m$-admissible} if $1\leq i_1<\cdots < i_{m_0}\leq m$. For any $m$-admissible $\un{i}\in\NN^{m_0}$ and $f\in\FF[\calV^{m_0}]^G$ we define the polynomial invariant  $f^{(\un{i})}\in \FF[\calV^m]^G$ as the result of the following substitutions in $f$: 
\[ x_{1,j} \to x_{i_1,j},\, \ldots, \, x_{m_0,j} \to x_{i_{m_0},j} \quad \text{ (for all } 1 \leq i \leq n).\]
Given a set $S \subset \FF[\calV^{m_0}]^G$, we define its {\it expansion} $S^{[m]} \subset \FF[\calV^m]^G$ by 
\begin{eq}\label{eq_expan}
S^{[m]} = \{ f^{(\un{i})} \,|\, f \in S \text{ and } \un{i}\in\NN^{m_0} \text{ is $m$-admissible}\}.
\end{eq}

%---------------------------------------------------------------------
\begin{remark}(see \cite[Remark~1.3]{Domokos_2007})\label{remark_expan}
Assume that $S_{1}$ and $S_{2}$ are separating sets for $\FF[\calV^{m_0}]^G$ and assume that $m>m_0$. Then $S_1^{[m]}$ is separating for $\FF[\calV^m]^G$ if and only if $S_2^{[m]}$ is separating for $\FF[\calV^m]^G$.
\end{remark}
\medskip

Denote by $\sisep(\FF[\calV],G)$ the minimal number $m_0$ such that the expansion of some separating set $S$ for $\FF[\calV^{m_0}]^G$ produces a separating set for $\FF[\calV^m]^G$ for all $m \geq m_0$. As an example, in~\cite{Lopatin_Reimers_2021} it was proven that $\sisep(\FF[\calV],\Sym_n) \leq \lfloor \frac{n}{2} \rfloor + 1$ over an arbitrary field $\FF$, where the symmetric group $\Sym_n$ acts on $\calV$ by the permutation of the coordinates. Moreover, $\sisep(\FF[\calV],\Sym_n)=\lfloor  \log_2(n)\rfloor + 1$ in case $\FF=\FF_2$ (see Corollary 4.12 of~\cite{Kemper_Lopatin_Reimers_2022}).

%---------------------------------------------------------------------
\begin{cor}\label{cor_main_si}
%Assume that the characteristic of $\FF_q$ is two.  Then 
We have
$$\sisep(\FF_q[V],\Op(\FF_q))=2.$$
\end{cor}
\begin{proof} For short, denote $\sisep=\sisep(\FF_q[V],\Op(\FF_q))$.
The upper bound $\sisep\leq 2$ follows from Theorem~\ref{theo_separ} and the fact that $\calT_2^{[m]}=\calT_m$, $(\calT_2^{(2)})^{[m]}=\calT_m^{(2)}$ for all $m>1$.

Assume $\sisep=1$. Then by Remark~\ref{remark_expan} and Lemma~\ref{lemma_separ1} we have that $\calT_1^{[m]}$ is a separating set. Since $\calT_1^{[m]}=\{N_1,T_1,\ldots,N_m,T_m\}$ is a proper subset of $\calT_m$ and $\calT_m^{(2)}$ for all $m>1$, we obtain a contradiction to Theorem~\ref{theo_separ}. 
\end{proof}

In case $p=2$ the next lemma follows from Theorem~\ref{theo_Chen} and in case $p>2$ it is well-known. We present to proof for the sake of completeness.  

%---------------------------------------------------------------------
\begin{lemma}\label{lemma_m1}
\tb{The algebra of invariants $\FF_q[V]^{\Op(\FF_q)}$ is generated by $N_1$ and $T_1$.}
\end{lemma}
\begin{proof} For short, denote $x=x_1$ and $y=y_1$. Consider an invariant $f\in \FF_q[V]^{\Op(\FF_q)}$. Then $f=\sum_{i=0}^k N_1^{i} h_i$, where $h_i$ is a linear combination of $\{x^r,y^s\,|\,r,s\geq0\}$. Since $N_1$ is the invariant, we obtain that $h_i$ is also an invariant. Considering the action of $\tau_{\al}$ on $h_i$, we can see that $h_i$ is a linear combination of $\{x^{(q-1)r},y^{(q-1)s}\,|\,r,s\geq0\}$. Moreover, considering the action of $\si_1$ on $h_i$, we can see that $h_i$ is a linear combination of $\{x^{(q-1)r}+y^{(q-1)r}\,|\,r\geq0\}$. Note that
$$x^{(q-1)r}+y^{(q-1)r} = T_1^r - f'$$
for some invariant $f'$. Applying the above reasoning to $f'$ and using induction by degree, we complete the proof.
\end{proof}

%---------------------------------------------------------------------
\begin{cor}\label{cor_main_be}  We have 
\begin{enumerate}
\item[$\bullet$] $\besep(\FF_q[V^m]^{\Op(\FF_q)})= 2$, in case $q=2$;

\item[$\bullet$] $\besep(\FF_q[V^m]^{\Op(\FF_q)})= q-1$, in case $q>2$;
\end{enumerate}
\end{cor}
\begin{proof} 
For short, we write $\besep(m)=\besep(\FF_q[V^m]^{\Op(\FF_q)})$. We also denote $\be=2$ in case $q=2$ and $\be=q-1$ in case $q>2$. The upper bound $\besep\leq \be$ follows from Theorem~\ref{theo_separ}. 

 \tb{Assume that $\besep(m)< \be$. Therefore, $\besep(1)<\be$. }

 \tb{Let $q=2$. By Lemma~\ref{lemma_m1} any invariant from $\FF_q[V]^{\Op(\FF_q)}$ of degree $<\be$ is a polynomial in $T_1$. Thus $\{T_1\}$ is separating set for $\FF_q[V]^{\Op(\FF_q)}$; a contradiction to Lemma~\ref{lemma_separ1}. }
 
\tb{Let $q>2$. By Lemma~\ref{lemma_m1} any invariant from $\FF_q[V]^{\Op(\FF_q)}$ of degree $<\be$ is a polynomial in $N_1$. Thus $\{N_1\}$ is separating set for $\FF_q[V]^{\Op(\FF_q)}$; a contradiction to Lemma~\ref{lemma_separ1}.}
 
% Denote by $S(m)$ the set of all invariants from $\calN \cup \calB \cup \calD$ of degree less than $\beta$. Since $\FF_q[V^m]^{\Op(\FF_q)}$ has the generating set $\calN \cup \calB \cup \calD$ consisting of homogeneous polynomials (see Theorem~\ref{theo_Chen}), any invariant $f$ is a polynomial in invariants of degree $\leq\deg(f)$. Therefore, any invariant of degree less than $\beta$ is a polynomial in elements of $S(m)$.  Thus $S(m)$ is a separating set for $\FF_q[V^m]^{\Op(\FF_q)}$. Since $(v,0,\ldots,0)$ and $(v',0,\ldots,0)$ are separated by $S(m)$ for all $v,v'\in V$ with $v\not\sim v'$, we obtain that $S(1)$ is a separating set for $\FF_q[V]^{\Op(\FF_q)}$. Since $S(1)=\{T_1\}$ in case $q=2$ and $S(1)=\{N_1\}$ in case $q>2$ (see Remark~\ref{remark_generators} for the details), we obtain that $S(1)$ is a proper subset of a minimal separating set $\calT_1=\calT_1^{(2)}=\{N_1,T_1\}$ (see Lemma~\ref{lemma_separ1}); a contradiction.
\end{proof}

%---------------------------------------------------------
\begin{cor}\label{cor_min_number_separ} \tb{
There exists a separating set for $\FF_q[V^m]^{\Op(\FF_q)}$ with $2m$ elements. On the other hand, 
$$|\calT_m^{(2)}|=\frac{1}{2}(m^2 + 3m) \quad \text{and} \quad |\calT_m|=m^2 + m.$$}
\end{cor}
\begin{proof}
By Theorem 1.1 from~\cite{Kemper_Lopatin_Reimers_2022}, the least possible number of elements of a separating set for $\FF_q[V^m]^{\Op(\FF_q)}$ is 
$$\ga_{\rm sep}=\lceil\log_q(\k)\rceil,$$
where the number $\k$ of $\Op(\FF_q)$-orbits on $V^m$ was explicitly described in Corollary~\ref{cor_orbit}. We have 
$$\k=\frac{q^{2m} +q^{m+1} - q^m + q -2 }{2(q-1)}.$$
Since $-q^m + q -2 \leq 0$, $\frac{1}{2(q-1)}\leq \frac{1}{q}$, and $q^{m+1}\leq q^{2m}$ for all $m\geq1$ and $q\geq 2$ we obtain that $\k \leq 2\, q^{2m-1}\leq q^{2m}$. Thus $\ga_{\rm sep}\leq 2m$. 
\end{proof}

%==========================================================================

%====================================================
\end{document}